\documentclass[11pt]{amsart}

\usepackage{amsmath}
\usepackage{amssymb,amsthm}
\usepackage{mathtools}
\usepackage[utf8]{inputenc}
\usepackage{xcolor}
\usepackage[margin=2.5cm]{geometry}
\usepackage{graphicx}
\usepackage[hidelinks]{hyperref}

\theoremstyle{definition}
\newtheorem{theorem}{Theorem}
\newtheorem{proposition}[theorem]{Proposition}
\newtheorem{lemma}[theorem]{Lemma}
\newtheorem{corollary}[theorem]{Corollary}

\newtheorem{remark}[theorem]{Remark}

\newcommand{\F}{\mathcal F}

\newcommand{\Lip}{\operatorname{Lip}}
\newcommand{\dist}{\operatorname{dist}}

\newcommand{\sgn}{\operatorname{sgn}}

\newcommand{\norm}[1]{\lVert#1\rVert}
\newcommand{\abs}[1]{\left\lvert#1\right\rvert}
\newcommand{\set}[1]{\left\{#1\right\}}
\newcommand{\angles}[1]{\left\langle#1\right\rangle}

\def\N{{\mathbb N}}
\def\R{{\mathbb R}}
\def\Z{{\mathbb Z}}

\title[Coarse-Lipschitz embedding into a separable dual]{A Coarse-Lipschitz Embedding of $c_0$\\
into a Separable Dual Banach Space}

\author{B\"unyamin Sar\i}
\email{bunyamin.sari@unt.edu}
\address{University of North Texas, Denton, TX}

\begin{document}

\maketitle

\begin{abstract}
We prove that $c_0$ admits a coarse Lipschitz embedding into a
separable dual Banach space and that the optimal coarse Lipschitz distortion is equal to two.  Let
$$
G=\mathbb Z^{<\omega}\subset c_0,
\qquad
G_R=G\cap R B_{c_0},
\quad R\in\mathbb N,
$$
with the metric inherited from $c_0$.  On each $G_R$ we construct a
commuting family of retractions onto finite initial segments of a special ordering of $G_R$, with
Lipschitz constant at most two.  Associated to these retractions there is a boundedly complete Schauder basis of $\mathcal F(G_R)$ whose
basis constant is at most two and which is $2R$-equivalent to the unit vector basis of $\ell_1$.  Consequently, each $\mathcal F(G_R)$ is
$2$-isomorphic to a separable dual Banach space, uniformly in $R$.

Kalton's annular decomposition then gives an embedding of
$\mathcal F(G)$ into a separable dual space with distortion at
most $2(1+\varepsilon)$ for every $\varepsilon>0$.  A decomposition
result of Aliaga and Medina further shows that
\[
\mathcal F(G)
\cong
\Big(
\bigoplus_{n\geq0}\mathcal F(G_{2^n})
\Big)_{\ell_1},
\]
and hence $\mathcal F(G)$ itself is isomorphic to a separable dual
Banach space.  The constant two is sharp: if $G_2$ embeds into $X^*$ with distortion strictly smaller than
two, then $X$ contains an isomorphic copy of $\ell_1$.  It follows that
the infimum of the coarse Lipschitz distortions of embeddings of
$c_0$ into separable dual Banach spaces is exactly two.
\end{abstract}

\section{Introduction}

Kalton proved that $c_0$ does not coarsely embed into a reflexive Banach space and, more generally, that if a Banach space coarsely contains $c_0$, then one of its iterated duals must be nonseparable \cite{Kalton2007}. Recall the well-known facts that $c_0$ does not linearly or bi-Lipschitz embed into a separable dual Banach space. This naturally raises the question whether $c_0$ can coarsely embed into a separable dual Banach space. Braga, Lancien, Petitjean, and Proch\'azka formulated and studied both the coarse and the coarse Lipschitz versions of this problem in \cite[Problem 1.2, 1.3]{BLPP}. They proved, in particular, that a coarse-Lipschitz embedding of $c_0$ into a separable dual space cannot have distortion strictly smaller than $3/2$. Further obstructions for several classes of separable dual spaces, including generalized James and James tree spaces, were obtained in \cite{JKS}.

 In this paper we answer both questions affirmatively. Moreover, we obtain the optimal quantitative result. If \[ \mathfrak D_{\mathrm{sd}}(c_0) = \inf\Big\{ \operatorname{dist}_{\mathrm{CL}}(f): \begin{array}{c} f:c_0\longrightarrow Y^*\text{ is a coarse-Lipschitz embedding},\\ Y^*\text{ is separable} \end{array} \Big\}, \] then \[ \mathfrak D_{\mathrm{sd}}(c_0)=2. \] Here $\operatorname{dist}_{\mathrm{CL}}(f)$ denotes the distortion of the coarse Lipschitz embedding $f$. 
 Since every coarse Lipschitz embedding is a coarse embedding, this also gives an affirmative answer to the coarse version of the problem. 
 
 The upper and lower estimates use different arguments. For the upper estimate we consider the integer grid \[ G=\mathbb Z^{<\omega}\subset c_0 \] and its bounded pieces \[ G_R=G\cap R B_{c_0}, \qquad R\in\mathbb N. \]
  The grid $G$ is a net in $c_0$, and hence $G$ and $c_0$ are coarse Lipschitz equivalent. The separable dual space which admits a coarse Lipschitz embedding of $c_0$ with distortion $\le 2(1+\varepsilon)$ turn out to be a space isomorphic to the Lipschitz-free space $\F(G)$. 
 
 For every $R$, we construct a Schauder basis of $\F(G_R)$ with basis constant at most two, independently of $R$. The construction is based on the {\em parent map} \[ (px)_i=\operatorname{sgn}(x_i)(|x_i|-1)_+, \] which moves every nonzero coordinate one unit toward zero. Its iterates give $G_R$ the structure of a rooted tree with root $0$ and height $R$. For every nonzero $x$, the pair $(px,x)$ is its {\em parent edge}, and the associated molecule \[ \delta(x)-\delta(px) \] is one of the basis vectors. A suitable enumeration of $G_R$ produces commuting retractions onto finite initial segments. All these retractions are $2$-Lipschitz, uniformly in $R$, and their linearizations are the partial sum projections of the parent edge basis. We also prove that this basis is $2R$-equivalent to the unit vector basis of $\ell_1$. In particular, it is boundedly complete. This yields, by a standard fact, that \[ \F(G_R) \] is $2$-isomorphic to a separable dual Banach space, with a constant independent of $R$. We then use Kalton's annular decomposition \cite{Kalton2004}. It gives, for every $\eta>0$, an almost isometric embedding \[ \F(G)\longrightarrow W:= \Big( \bigoplus_{n\geq0}\F(G_{2^n}) \Big)_{\ell_1}. \] Thus we obtain one separable dual Banach space $Z$, $2$-isomorphic to $W$, such that, for every $\eta>0$, there is a linear embedding \[ T_\eta:\F(G)\longrightarrow Z \] satisfying \[ \|\mu\| \leq \|T_\eta\mu\| \leq 2(1+\eta)\|\mu\|. \] Since $G$ is a net in $c_0$ this gives a coarse Lipschitz embedding of $c_0$ into $Z$. 
 
 It was pointed out to us by R\'uben Medina that the work of Aliaga and Medina \cite{AliagaMedina2026}, together with Kalton's decomposition and Pe{\l}czy\'nski's decomposition method, gives \[ \F(G) \cong \Big( \bigoplus_{n\geq0}\F(G_{2^n}) \Big)_{\ell_1}. \] 
It follows that $\F(G)$ itself is isomorphic to the separable dual space $Z$; see Corollary~\ref{cor:FG-separable-dual}.

For the lower distortion estimate, we prove a sharp local obstruction. If \[ f:G_2\longrightarrow X^* \] is a bi-Lipschitz embedding with distortion strictly smaller than two, then $X$ contains an isomorphic copy of $\ell_1$; see Proposition~\ref{prop:radius-two}. In particular, $G_2$ cannot embed with distortion strictly smaller than two into a separable dual Banach space. Applying this result to large dilates of $G_2$ shows that every coarse Lipschitz embedding of $c_0$ into a separable dual space has distortion at least two. Together with the construction above, this proves \[ \mathfrak D_{\mathrm{sd}}(c_0)=2. \] 

The scheme of our construction is related to the work of Garc\'ia-Lirola, Petitjean, Proch\'azka, and Rueda Zoca \cite{GLPPRZ}, and to that of Braga, Lancien, Petitjean, and Proch\'azka \cite{BLPP}. The former authors gave a criterion identifying certain Lipschitz-free spaces as dual spaces. The latter applied this criterion to bounded pieces of Kalton's interlaced graphs and then used Kalton's decomposition to place the full graph in a countable $\ell_1$-sum of separable dual spaces. One could use the retractions constructed here to recover that that duality criterion; see Remark~\ref{rem:compact-variant}. However, our argument used in the main proof is more direct and gives the uniform constant two immediately. 

Our construction is also related to the theory of retractional bases for Lipschitz-free spaces developed in \cite{HN,HM}. H\'ajek and Novotn\'y \cite{HN} had already constructed a monotone Schauder basis for the free space over the full integer grid $G=\mathbb Z^{<\omega}\subset c_0$, and H\'ajek and Medina \cite{HM} obtained Schauder bases in free spaces over many nets in Banach spaces. The feature needed here is different: the
construction is carried out uniformly on the bounded grids $G_R$,
and an explicit boundedly complete
basis with basis constant at most two is constructed, independently of $R$.  This
uniform control is what makes the passage to one separable dual
target possible.

\medskip
\noindent\textbf{Acknowledgments.}
We thank R\'uben Medina for pointing out
Corollary~\ref{cor:FG-separable-dual}.

\medskip
\noindent\textbf{Use of generative AI.}
OpenAI's ChatGPT was used during the exploration of the problem,
including in developing an initial proof based on the scheme described in
Remark~\ref{rem:compact-variant}. The proof presented here is more direct and was written by the author who takes full responsibility for the content.

\section{Preliminaries}
We recall briefly standard terms and facts used in the paper. For more background, see, for instance, \cite{BL}.
\subsection{The integer grid and coarse-Lipschitz maps }
Let
$$
G=\Z^{<\omega},\ \ G_R=\{x\in G:\|x\|_{\infty}\le R\};
$$
with the metric $d(x,y)=\|x-y\|_{\infty}$ inherited from $c_0$.  

For a map $f:X\to Y$ between Banach (or metric) spaces, we say that $f$ is a {\em coarse-Lipschitz embedding} if there are $a, L>0$ and $b, B\ge 0$ such that

\begin{equation}\label{eq:cl-bounds}
a\|x-y\|-b\le \|f(x)-f(y)\|\le L\|x-y\|+B\ \ (x, y\in X).
\end{equation}

Its coarse-Lipschitz distortion is 
$$
\dist_{CL}(f)=\inf\set{\frac{L}{a}: a, L, b, B \text{ satisfy  \eqref{eq:cl-bounds}}}
$$

This is the large scale analog of bi-Lipschitz distortion. 

$G$ and $c_0$ are coarse-Lipschitz equivalent. The map $Q:c_0\to G$ that rounds each coordinate to a nearest integer gives
$$
\|x-y\|_{\infty}-1\le d(Qx, Qy)\le \|x-y\|_{\infty}+1.
$$

\subsection{Lipschitz-free spaces}
Let $(M,d,0)$ be a metric space with a distinguished point 0. The Banach space $\Lip_0(M)$ consists of all real valued Lipschitz functions on $M$ that vanish at 0, equipped with
$$
\norm{f}_{\Lip}=\sup_{x\neq y}\frac{|f(x)-f(y)|}{d(x,y)}.
$$
For $x\in M$, evaluation at $x$ defines $\delta_M(x)\in \Lip_0(M)^*$ by
$\langle \delta_M(x), f\rangle=f(x).$ The {\em Lipschitz-free space} over $M$ is
$$
\F(M)=\overline{\operatorname{span}}\set{\delta_M(x): x\in M}\subseteq \Lip_0(M)^*.
$$

The Dirac map $\delta_M: M\to \F(M)$ is an isometric embedding:
$$
       \norm{\delta_M(x)-\delta_M(y)}_{\F(M)}=d(x,y).
$$
These spaces linearize Lipschitz maps in the following sense \cite{GK}. If $f:(M,0)\to (N,0)$ is Lipschitz, there is a unique bounded linear operator
$$
\widehat f:\F(M)\to \F(N),\ \ \ \widehat f \delta_M(x)=\delta_N(f(x)),
$$
and $\|\widehat f\|=\Lip(f)$. Thus a bi-Lipschitz bijection $f:M\to N$ linearizes to an isomorphism $\widehat f:\F(M)\to \F(N)$ with 
$$
\norm{\widehat f}\norm{\widehat f^{-1}}\le \dist(f).
$$
If $A\subseteq M$ contains 0, then the canonical map $\F(A)\to \F(M)$ is an isometric embedding.

\subsection{Retractional bases}
A $K$-{\em retractional basis} on a countable pointed metric space $M=\set{0, x_1, x_2, \ldots}$ is a sequence of $K$-Lipschitz retractions $r_n:M\to A_n=\set{0, x_1,\ldots, x_n}$ satisfying
$$
r_mr_n=r_nr_m=r_{\min\{m,n\}}.
$$

The linearizations $\widehat r_n$ are then the partial sum projections of a Schauder basis of $\F(M)$ with constant $K$ \cite[Theorem 13]{HN}. This notion was developed in \cite{HN} and further studied in \cite{HM}.

Below we will define a parent map $p$ on $G_R$ which takes every coordinate one step closer to the zero, and thus defines a rooted tree on $G_R$ of height $R$. We will consider vectors (molecules)
$$
\delta(x_n)-\delta(px_n)
$$
which can be thought of an edge on the tree and hence we will call them {\em edge vectors}.

\section{A finite height retraction system}

Fix $R\in\N$. Define the parent map $p:G_R\to G_R$ coordinatewise by
$$
(px)_i=\sgn(x_i)(|x_i|-1)_+.
$$
Then $p(0)=0$, and
\begin{equation}\label{eq:parent-properties}
 \norm{px-py}_\infty\leq\norm{x-y}_\infty,
 \qquad
 \norm{x-px}_\infty=1\quad(x\neq0),
 \qquad
 \norm{px}_\infty=\norm{x}_\infty-1\quad(x\neq0).
\end{equation}

A useful interpretation is that the directed graph with vertex set $G_R$ and an edge from $px$ to $x$ for
every $x\neq0$ is a countably branching rooted tree with root $0$.  The unique
path from $0$ to $x$ has length $\norm{x}_\infty$.  Thus the tree has height
$R$, and $p^R x=0$ for every $x\in G_R$.

For $x,y\in G_R$, write $y\preceq x$ if for every coordinate $i$, either
$y_i=0$, or $x_i y_i>0$ and $\abs{y_i}\leq\abs{x_i}$.  We define {\em the down set} of $x$ as
$$
       D(x)=\set{y\in G_R:y\preceq x},
$$
which is finite.

\begin{lemma}\label{lem:predecessor}
If $a, z\in G_R$ and $d(a,z)\le 1$, then $pz\preceq a$. In particular, 
$$
\set{pz: z\in G_R, d(a,z)\le 1}\subseteq D(a).
$$
\end{lemma}

\begin{proof}
Fix a coordinate $i$. If $a_i=0$, then $z_i\in \set{-1, 0, 1}$ and hence $(pz)_i=0$. If $a_i>0$, then $z_i\in \set{a_i-1, a_i, a_i+1}$, so
$$
(pz)_i\in\set{(a_i-2)_+, a_i-1, a_i}.
$$
These numbers are nonnegative and at most $a_i$. The case $a_i<0$ is similar.
\end{proof}
\noindent{\bf Enumeration.} Choose an enumeration 

\begin{equation}\label{eq:enumeration}
       G_R=\set{0,x_1,x_2,\ldots}
\end{equation}
with the property
\begin{equation}\label{eq:downset-order}
 y\in D(x_n)\setminus\set{x_n}
 \quad\Longrightarrow\quad
 y\in\set{0,x_1,\ldots,x_{n-1}}.
\end{equation}
Such an enumeration exists. Start with an arbitrary enumeration $(q_j)$ of $G_R$. At each stage $j$, add all points of $D(q_j)$ which have not yet been listed, in nondecreasing order of their $\ell_1$-norm $\|x\|_1=\sum_i |x_i|$ (and choose arbitrarily if their norms are equal). This works because if $y\preceq x$ and $x\neq y$, then $\|y\|_1<\|x\|_1$. Therefore, every proper predecessor of a point is listed before that point. Put

$$
A_0=\set{0},\ \ \ A_n=\set{0, x_1, \ldots, x_n}\ \ (n\ge 1).
$$
Then every $A_n$ is downward closed: if $x\in A_n$ and $y\preceq x$, then
$y\in A_n$.  Moreover, every $q_j$ is eventually inserted.

\noindent{\bf Retractions.} For $x\in G_R$, define
\begin{equation}\label{eq:retraction}
       k_n(x)=\min\set{k\geq0:p^k x\in A_n},
       \qquad
       r_n(x)=p^{k_n(x)}x.
\end{equation}
The minimum exists because $p^R x=0$.  Moreover,
$p(x_n)\in A_{n-1}$.

The map $r_n$ fixes $A_n$, so it is a retraction. Moreover, it is clear from the definition that
\begin{equation}\label{eq:commuting}
       r_mr_n=r_nr_m=r_{\min\set{m,n}}
       \qquad(m,n\geq0).
\end{equation}

We have

\begin{proposition}\label{prop:two-lipschitz}
Every $r_n$ is $2$-Lipschitz.  If $R=1$, every $r_n$ is $1$-Lipschitz.
\end{proposition}

\begin{proof}
First suppose $d(x,y)\le 1$. Put $i=k_n(x)$ and $j=k_n(y)$, and assume $i\le j$. If $i=j$, the first inequality in
\eqref{eq:parent-properties} gives
$$
       d(r_nx,r_ny)=d(p^ix,p^iy)\leq1.
$$
Suppose $i<j$. Let $a=p^ix\in A_n$ and $z=p^iy$. Then $d(a,z)\le 1$. By Lemma~\ref{lem:predecessor}, $pz\in D(a)\subseteq A_n$. So we must have $j=i+1$. Thus
$$
d(r_nx, r_ny)=d(a,pz)\le d(a,z)+d(z,pz)\le 2.
$$

For arbitrary $x,y$, put $k=d(x,y)\in\N\cup\set{0}$.  There is a path
$x=z_0,z_1,\ldots,z_k=y$ in $G_R$ with
$d(z_{\ell-1},z_\ell)\leq1$: at each step, move every coordinate which has
not reached its target one unit toward that target.  Applying the preceding
estimate along this path yields
$$
       d(r_nx,r_ny)\leq2d(x,y).
$$

If $R=1$, every nonzero point has parent zero.  The map $r_n$ fixes the
listed points and sends each unlisted point to zero.  Since distinct points of
$G_1$ are at distance at least one and every nonzero point has norm one, this
map is $1$-Lipschitz.

\end{proof}

\section{Edge basis equivalent to $\ell_1$}

Fix $R\in\N$, and for $n\ge 1$, let

$$P_n=\widehat r_n:\F(G_R)\to \F(G_R).$$

By Proposition~\ref{prop:two-lipschitz},
$$
\sup_n\norm{P_n}\leq2,
$$
and the supremum is one when $R=1$.

Since $G_R$ is bounded and uniformly discrete, it is well-known that $\F(G_R)$ is isomorphic to $\ell_1$ (cf. \cite{Weaver, Kalton2004}). The point of the following is that the sequence of edge vectors $(\delta(x_n)-\delta(px_n))_n$ form a Schauder basis with constant $2$ {\em independent} of $R$, and it is $2R$-equivalent to the unit vector basis of $\ell_1$. This is a particularly simple case of the retractional bases constructed for several classes of grids in \cite{HN, HM}.

\begin{proposition}\label{prop:parent-edge-basis}
With respect to the enumeration \eqref{eq:enumeration}, the edge vectors (molecules)
$$
b_n=\delta(x_n)-\delta(px_n),
\qquad n\geq1,
$$
form a normalized Schauder basis of $\F(G_R)$ with basis constant at most two.  When $R=1$,
the basis is monotone.  Moreover, for every finite scalar sequence
$(a_j)$, we have
\begin{equation}\label{eq:l1-equivalence}
\frac{1}{2R}\sum_j\abs{a_j}
\leq\Big\|\sum_j a_jb_j\Big\|
\leq\sum_j\abs{a_j}.
\end{equation}
\end{proposition}

\begin{proof}
Since $px_j\in A_{j-1}$, $b_j$ is the molecule associated with the parent edge $(px_j, x_j)$. That is, it is the new edge introduced when $x_j$ is added the to $A_{j-1}$ (this is the reason for the formula $b_j=\delta(x_j)-\delta(px_j)$). We check that 
$$
P_nb_j=
\begin{cases}
b_j,&j\leq n,\\
0,&j>n.
\end{cases}
$$
If $j\leq n$, then $r_n$ fixes both $x_j$ and $px_j$.  Suppose that $j>n$,
and let
$$
k=\min\set{l\geq1:p^lx_j\in A_n}.
$$
Then $r_nx_j=p^kx_j$.  The first point on the parent chain of $px_j$ which
belongs to $A_n$ is the same point, since
$$
p^{k-1}(px_j)=p^kx_j.
$$
Thus $r_n(px_j)=r_nx_j$, and consequently $P_nb_j=0$. Thus $P_n$'s are $n$th partial sum operators, and recall that $\norm{P_n}\le 2$.

Note that
$$
\delta(x_j)=\sum_{l=0}^{\norm{x_j}_\infty-1}
\bigl(\delta(p^lx_j)-\delta(p^{l+1}x_j)\bigr),
$$
and every summand on the right is one of the vectors $b_i$ with $i\leq j$.
Therefore
$$
\operatorname{span}\set{b_1,\ldots,b_n}
=\operatorname{span}\set{\delta(x_1),\ldots,\delta(x_n)}
$$
for every $n$.  In particular, $(b_j)$ is linearly independent and has dense
linear span in $\F(G_R)$.

It remains to prove \eqref{eq:l1-equivalence}.  Let $(a_j)$ be finitely
supported and write
$$
\mu=\sum_j a_jb_j=\sum_k c_k\delta(x_k).
$$
Since $G_R$ is $1$-separated, the function $f:G_R\to\R$ defined by
$$
f(0)=0,
\qquad
f(x_k)=\frac12\sgn(c_k)
$$
is $1$-Lipschitz.  Hence
$$
\norm{\mu}\geq\angles{\mu,f}=\frac12\sum_k\abs{c_k}.
$$
On the other hand, each Dirac vector $\delta(x_k)$ is the sum of the
parent-edge molecules along its path from $0$ to $x_k$, and this path has
length at most $R$.  Expanding each $\delta(x_k)$ in this way and comparing
the unique $(b_j)$-expansions of $\mu$ gives
$$
\sum_j\abs{a_j}\leq R\sum_k\abs{c_k}\leq2R\norm{\mu}.
$$
Since $\norm{b_j}=d(x_j,px_j)=1$, we also have
$$
\norm{\mu}\leq\sum_j\abs{a_j}.
$$
This proves \eqref{eq:l1-equivalence}.

\end{proof}

Since the basis is equivalent to the unit vector basis of $\ell_1$ it is boundedly complete. The basis constant $2$ implies that $\F(G_R)$ is $2$-isomorphic to the dual of the  space $[b^*_n]$ (c.f., Proposition 1.b.4 of \cite{LT}). Since this is critical for the main result we record this as

\begin{corollary}\label{cor:local-upper}
For every $R\geq1$, the space $\F(G_R)$ is $2$-isomorphic to a separable dual space. When $R=1$, it is isometric to a separable dual space.\end{corollary} 

\section{The main theorem}
The main theorem now follows easily from Kalton's decomposition theorem (Proposition 4.3 of \cite{Kalton2004}). For a pointed metric space $M$, write
$$
       M_k=\set{x\in M:d(x,0)\leq2^k}
       \qquad(k\in\Z).
$$
Kalton's annular decomposition states that, for every $\eta>0$, there is a
linear embedding
\begin{equation}\label{eq:kalton-decomp}
 S_\eta:\F(M)\longrightarrow
       \left(\bigoplus_{k\in\Z}\F(M_k)\right)_{\ell_1}
\end{equation}
with
\begin{equation}\label{eq:kalton-bounds}
       \norm{\mu}\leq\norm{S_\eta\mu}
       \leq(1+\eta)\norm{\mu}
       \qquad(\mu\in\F(M)).
\end{equation}

Thus we have almost isometric embeddings
$$
\F(G)\longrightarrow 
W:=\left(\bigoplus_{n\geq0}\F(G_{2^n})\right)_{\ell_1}.
$$

For each $n\geq0$, let $(b_{n,j})_{j\geq1}$ denote the boundedly complete basis of $\F(G_{2^n})$ constructed above, and let \[ E_n= \overline{\operatorname{span}}\{b_{n,j}^*:j\geq1\} \subseteq \F(G_{2^n})^* \] be its  predual. By Corollary~\ref{cor:local-upper}, the canonical map from $\F(G_{2^n})$ onto $E_n^*$ is an isomorphism with distortion at most two. Put \[ E=\left(\bigoplus_{n\geq0}E_n\right)_{c_0}, \qquad Z=E^* = \left(\bigoplus_{n\geq0}E_n^*\right)_{\ell_1}. \] Then $Z$ is a separable dual Banach space and \[ W= \left(\bigoplus_{n\geq0}\F(G_{2^n})\right)_{\ell_1} \] is $2$-isomorphic to $Z$.

Thus we have

\begin{theorem}\label{thm:global-free}
There is a separable dual Banach space $Z$ such that, for every $\eta>0$,
there is a linear embedding
$$
       T_\eta:\F(G)\longrightarrow Z
$$
with
\begin{equation}\label{eq:global-free-bounds}
       \norm{\mu}\leq\norm{T_\eta\mu}
       \leq2(1+\eta)\norm{\mu}.
\end{equation}
In particular, the infimum of the coarse-Lipschitz distortions of embeddings of $c_0$ into separable dual spaces is at most $2$.
\end{theorem}

It was pointed out to us by Rub\'en Medina that the spaces $\mathcal F(G)$ and $W$ are actually isomorphic.

\begin{corollary}\label{cor:FG-separable-dual} $\F(G)$ is isomorphic to a separable dual Banach space. \end{corollary}

\begin{proof}
This follows from the work of Aliaga and Medina
\cite{AliagaMedina2026}:

Let $S$ be a
subset of a metric space $M$ containing distinguished point 0.  $S$ is said to be \emph{weak$^*$
$C$-Lipschitz extendable in $M$} if there exists a bounded linear
extension operator
\[
E:\Lip_0(S)\longrightarrow \Lip_0(M)
\]
such that
\[
(Ef)|_S=f,\qquad \|E\|\leq C,
\]
and $E$ is pointwise-to-pointwise continuous.  Equivalently, $E$ is weak$^*$--weak$^*$
continuous. $M$ is called {\em homogeneous} if for any $x, y\in M$ there is a bijective isometry $\varphi$ on $M$ such that $\varphi(x)=y$.

The space $G$ is homogeneous and unbounded,
and its closed balls are uniformly weak$^*$ Lipschitz extendable.
Indeed, for every $R\in\mathbb N$ the coordinatewise clipping map
$$
q_R:G\longrightarrow G_R,\qquad
(q_Rx)_i=\operatorname{sgn}(x_i)\min\{|x_i|,R\},
$$
is a $1$-Lipschitz retraction.  By translation, the same is true for
every closed ball in $G$.  Thus by (the proof of) \cite[Proposition 3.1]{AliagaMedina2026} we have a complemented embedding
 
$$
\left(
\bigoplus_{n=0}^{\infty}
\mathcal F(G_{2^n})
\right)_{\ell_1}
\xhookrightarrow{c}
\mathcal F(G).
$$

We also have a complemented embedding by \cite[Lemma 1.2 (Kalton)]{AliagaMedina2026},
$$
\mathcal F(G)
\xhookrightarrow{c}
\left(
\bigoplus_{n=0}^{\infty}
\mathcal F(G_{2^n})
\right)_{\ell_1}.
$$

Moreover, by \cite[Theorem~3.2]{AliagaMedina2026}, $\F(G)\cong \left(
\bigoplus_{n}
\mathcal F(G)
\right)_{\ell_1}.$ (We remark that this also follows from \cite[Theorem 8]{HN}.)
Thus Pe{\l}czy\'nski's decomposition method yields that they are isomorphic. Thus $\F(G)\cong W\cong Z$.
\end{proof}

\section{The sharp radius-two obstruction}\label{sec:lower}
In \cite{BLPP} it was shown that $c_0$ does not coarse-Lipschitz embed into a separable dual space with distortion strictly less than $3/2$. We show that the sharp constant is $2$. 
\begin{proposition}\label{prop:radius-two}
Let $X$ be a Banach space and suppose that $f:G_2\to X^*$ satisfies
\begin{equation}\label{eq:bilip-G2}
 d(u,v)\leq\norm{f(u)-f(v)}\leq Dd(u,v)
 \qquad(u,v\in G_2)
\end{equation}
for some $D<2$.  Then $X$ contains an isomorphic copy of $\ell_1$.
\end{proposition}

\begin{proof}
Choose $\varepsilon>0$ so that
$$
       \alpha=4-\varepsilon-2D>0.
$$
For each $k$, choose $x_k\in S_X$ with
\begin{equation}\label{eq:norming-pair}
       \angles{x_k,f(2e_k)-f(-2e_k)}>4-\varepsilon.
\end{equation}
For $A\subseteq\N$ and $m\in\N$, put
$$
 u_m^A=\sum_{j=1}^m\varepsilon_j^A e_j,
 \qquad
 \varepsilon_j^A=
 \begin{cases}
  1,&j\in A,\\
 -1,&j\notin A.
 \end{cases}
$$
Thus $u_m^{A^c}=-u_m^A$, and both points belong to $G_1$.

Fix $m\geq k$.  If $k\in A$, then
$$
       d(u_m^A,2e_k)=1,
       \qquad
       d(u_m^{A^c},-2e_k)=1.
$$
Decomposing through these two endpoints and using
\eqref{eq:bilip-G2}--\eqref{eq:norming-pair}, we obtain
\begin{equation}\label{eq:positive-sign}
       \angles{x_k,f(u_m^A)-f(u_m^{A^c})}\geq\alpha.
\end{equation}
If $k\notin A$, the same argument with $A$ and $A^c$ interchanged gives
\begin{equation}\label{eq:negative-sign}
       \angles{x_k,f(u_m^A)-f(u_m^{A^c})}\leq-\alpha.
\end{equation}

Fix a nonprincipal ultrafilter $\mathcal U$ on $\N$ and set
$$
       z_A^*=w^*-\lim_{m,\mathcal U}
             \bigl(f(u_m^A)-f(u_m^{A^c})\bigr).
$$
The sequence under the limit is bounded, since
$d(u_m^A,u_m^{A^c})\leq2$.  Equations
\eqref{eq:positive-sign} and \eqref{eq:negative-sign} imply
\begin{equation}\label{eq:sign-separation}
 \angles{x_k,z_A^*}\geq\alpha\quad(k\in A),
 \qquad
 \angles{x_k,z_A^*}\leq-\alpha\quad(k\notin A).
\end{equation}
No subsequence of $(x_k)$ is weakly Cauchy: given a subsequence, choose $A$ to
alternate on its indices and use \eqref{eq:sign-separation}.  Rosenthal's
$\ell_1$ theorem \cite{Rosenthal} therefore yields a subsequence equivalent to
the unit vector basis of $\ell_1$.
\end{proof}

\begin{corollary}\label{cor:G2-sep-dual}
The metric space $G_2$ does not bi-Lipschitz embed with distortion strictly less than two
into any separable dual Banach space.
\end{corollary}

Put
$$
\mathfrak D_{\mathrm{sd}}(c_0)
=\inf\left\{\dist_{\mathrm{CL}}(f):
\begin{array}{c}
f:c_0\to Y^*\text{ is a coarse-Lipschitz embedding,}\\
Y^*\text{ is separable}
\end{array}
\right\}.
$$
Then
$$
       \mathfrak D_{\mathrm{sd}}(c_0)=2.
$$

The upper estimate follows from Theorem~\ref{thm:global-free}.

For the lower estimate, suppose that $f:c_0\to X^*$ is a coarse-Lipschitz
embedding, $X^*$ is separable, and for some constants in
\eqref{eq:cl-bounds} one has $L/a<2$.  For $t>0$, define
$$
       f_t(u)=\frac{f(tu)-f(0)}{t}
       \qquad(u\in G_2).
$$
If $u\neq v$, then $d(u,v)\geq1$, and
$$
 (a-b/t)d(u,v)
 \leq\norm{f_t(u)-f_t(v)}
 \leq(L+B/t)d(u,v).
$$
Choose sufficiently large $t$ with $a-b/t>0$ and
$$
       \frac{L+B/t}{a-b/t}<2.
$$

Then 
$$
\frac{1}{a-b/t}f_t:G_2\to X^*
$$
is a bi-Lipschitz embedding of distortion strictly smaller than 2, contradicting
Corollary~\ref{cor:G2-sep-dual}.  Therefore every such embedding has
coarse-Lipschitz distortion at least two.

\section{Further Remarks} We briefly point out some further consequences.

\begin{remark}\label{cor:uniformly-discrete-free} The space $\F(G)$ is already a natural object in nonlinear Banach space theory (see \cite[Section 4.5]{BL} for an overview). It is universal for separable metric spaces and coarse Lipschitz embeddings. The conclusion above says more specifically that {\em this canonical coarse universal free space is isomorphic to a separable dual Banach space.} The same observation has a linear consequence for a much larger class of free spaces. Every separable uniformly discrete metric space $M$ admits, for every $\varepsilon>0$, a bi-Lipschitz embedding into $G$ with distortion smaller than $2+\varepsilon$. Linearizing this embedding and then using the map from $\F(G)$ into $Z$, we obtain \[ \F(M)\longrightarrow Z \] with distortion arbitrarily close to four. Thus one fixed separable dual space linearly contains an isomorphic copy of $\F(M)$ for every separable uniformly discrete metric space $M$. 
\end{remark}

\begin{remark}\label{cor:optimal-szlenk}The construction is also optimal from the point of view of the Szlenk index. It follows from \cite[Theorem 4.3]{BLPP} that if a separable $X^*$ contains coarse Lipschitz copy of $c_0$, then $\operatorname{Sz}(X)>\omega.$ Thus $\omega^2$ is the least possible Szlenk index of a predual of such a target.
Let \[ E= \Big( \bigoplus_{n\geq0}E_n \Big)_{c_0} \] be the natural predual of $Z$, where $E_n$ is the predual of $\F(G_{2^n})$. Since for each $n$, $[b^*_{n,j}]_j$ is equivalent to $c_0$ basis, $E_n=[b^*_{n,j}]_j$ is isomorphic to $c_0$, and Brooker's direct-sum theorem \cite{Brooker} gives $\operatorname{Sz}(E)\leq\omega^2.$ Thus $\operatorname{Sz}(E)=\omega^2.$
\end{remark}

\begin{remark}\label{rem:compact-variant}
The retractions we consider in this paper also lead to a compact topology assumption needed in the criterion of \cite{GLPPRZ} to obtain an isometric dual space and recovers the compact-topological dualization scheme used in \cite{BLPP}.  Indeed, the map
$$
x\longmapsto(r_nx)_{n\geq0}
$$
identifies $G_R$ with a compact subset of $\prod_n A_n$, while the envelope
metric
$$
\widetilde d_R(x,y)=\sup_n d(r_nx,r_ny)
$$
is lower semicontinuous and satisfies
$$
d(x,y)\leq\widetilde d_R(x,y)\leq2d(x,y).
$$
We omit the details, since this variant is not used here.
\end{remark}

\begin{remark}\label{rem:godun-dual-distance}
The uniform constant $2$ obtained above is noteworthy in relation to a
classical quantitative renorming problem of Godun.  For a Banach space
$X$, let
$$
D_{\mathrm{dual}}(X)
=\sup_{\rho\sim\|\cdot\|}
\inf\bigl\{
d_{BM}\bigl((X,\rho),Y^*\bigr):
Y\text{ is a Banach space}
\bigr\},
$$
where the supremum is taken over all equivalent norms $\rho$ on $X$.
Godun proved that
$$
D_{\mathrm{dual}}(X)>2
$$
for every nonreflexive Banach space $X$; see \cite{Godun} and
\cite{S}.  In particular,
there are equivalent norms on $\ell_1$ whose Banach--Mazur distance from
every dual Banach space is strictly greater than $2$.  By contrast, the
particular renormings of $\ell_1$ arising here as $\F(G_R,d)$ have
distance at most $2$ from the class of dual spaces, uniformly in $R$.
For $R\geq2$, Proposition~\ref{prop:radius-two} shows that this distance
is exactly $2$, even though the equivalence of
$\F(G_R,d)$ with the standard $\ell_1$ norm has constant $2R$.

\end{remark}

\end{document}